\documentclass[12pt]{amsart}

\usepackage[inner=25mm, outer=25mm, textheight=225mm]{geometry}

\usepackage{graphicx}
\usepackage{color}

\usepackage{amsmath,graphics,color,enumitem}
\usepackage{amssymb}
\usepackage{xypic,hyperref}

\theoremstyle{plain}
\newtheorem*{theorem*}{Theorem}
\newtheorem*{lemma*} {Lemma}
\newtheorem*{corollary*} {Corollary}
\newtheorem*{proposition*} {Proposition}
\newtheorem{theorem}{Theorem}[section]
\newtheorem{lemma}[theorem]{Lemma}
\newtheorem{corollary}[theorem]{Corollary}
\newtheorem{proposition}[theorem]{Proposition}
\newtheorem{conjecture}[theorem]{Conjecture}
\newtheorem{question}[theorem]{Question}

\theoremstyle{remark}
\newtheorem*{remark}{Remark}
\newtheorem*{definition}{Definition}

\theoremstyle{definition}

\newcounter{commentcounter}

\def\be{\begin{equation}}
\def\ee{\end{equation}}
\def\co{\colon}
\def\RR{\mathcal{R}}

\def \K {\mathbf{K}}\def \Z {\mathbf{Z}}\def \C {\mathbf{C}}
\def \G {\mathbf{G}}

\def\R{\mathbb{R}}

\def\K{\mathbb{K}}
\def\id{\op{id}}
\def\Z{\mathbb{Z}}\def\C{\mathbb{C}}

\def\N{\mathbb{N}}
\def\part{\partial}\def\ll{\langle}\def\rr{\rangle}

\def\bp{\begin{pmatrix}}
\def\sm{\setminus}\def\ep{\end{pmatrix}}
\def\bn{\begin{enumerate}}

\def\rank{\op{rank}}
\def\en{\end{enumerate}}\def\ba{\begin{array}}
\def\ea{\end{array}}

\def\S{\Sigma}

\def\fr12{\frac{1}{2}}
\def\ker{\op{Ker}}

\def\hom{\op{Hom}}

\def\zt{\Z\tpm}

\def\G{\Gamma}
\def\ol{\overline}

\def\op{\operatorname}

\def\zt{\Z[t^{\pm 1}]}

\def\K{\mathbb{K}}

\def\th{\op{th}}

\def\Wh{\op{Wh}}

\def\cmtbf#1{} \def\cmt#1{}

\def\SS{\mathcal{S}}
\def\NN{\mathcal{N}}
\def\PP{\mathcal{P}}

\def\TT{\mathcal{T}}
\def\MM{\mathcal{M}}
\def\QQ{\mathcal{Q}}

\def\VV{\mathcal{V}}

\def\DD{\mathcal{D}}
\def\RR{\mathcal{R}}

\def\wti{\widetilde}

\def\sym{{\op{sym}}}
\def\mfp{\mathfrak{P}}
\def\mfg{\mathfrak{G}}
\def\TT{\mathcal{T}}
\def\th{\op{th}}

\def\whac{\op{Wh-AC}}

\newcommand{\smsum}[2]{\mbox{\footnotesize$\displaystyle\sum\limits_{#1}^{#2}$}} % small medium sum
\newcommand{\tmsum}[2]{\mbox{$\textstyle \sum\limits_{#1}^{#2}$}} % tiny medium sum

\begin{document}

\title[Groups and polytopes]{Groups and polytopes}

\author{Stefan Friedl}
\address{Fakult\"at f\"ur Mathematik\\ Universit\"at Regensburg\\   Germany}
\email{sfriedl@gmail.com}

\author{Wolfgang L\"uck}
        \address{Mathematisches Institut der Universit\"at Bonn\\
                Endenicher Allee 60\\
                53115 Bonn, Germany}
         \email{wolfgang.lueck@him.uni-bonn.de}
          \urladdr{http://www.him.uni-bonn.de/lueck}
          
\author{Stephan Tillmann}
\address{School of Mathematics and Statistics\\ The University of Sydney\\ NSW 2006 Australia} 
\email{stephan.tillmann@sydney.edu.au} 

\date{\today}
\def\subjclassname{\textup{2000} Mathematics Subject Classification}
\expandafter\let\csname subjclassname@1991\endcsname=\subjclassname \expandafter\let\csname
subjclassname@2000\endcsname=\subjclassname 

\date{\today}
\begin{abstract} 
In a series of papers  the authors associated to an $L^2$-acyclic group $\Gamma$ an invariant $\PP(\Gamma)$ that is a  formal difference of polytopes in the vector space $H_1(\Gamma;\R)$. This invariant is in particular defined for most 3-manifold groups, for most 2-generator 1-relator groups and for all free-by-cyclic groups. In most of the above cases the invariant can be viewed as an actual polytope.

In this survey paper we will recall the definition of the polytope invariant $\PP(\Gamma)$ and we state some of the main properties. We conclude with a list of open problems.

\end{abstract}
\maketitle

%==================================================================
\section{Introduction}

%==================================================================
\subsection{The Grothendieck group of  polytopes}
\label{section:minkowski}
A \emph{polytope in  a finite dimensional real vector space $V$} is defined as the convex
hull of a finite non-empty subset of $V$.  Given a polytope $\PP$ we denote by
\[ \ol{\PP}\,\,:=\,\,\{-x\,|\,x\in \PP\}\]
the mirror image of $\PP$ in the origin.\footnote{In the literature the mirror image of $\PP$ in the origin is often denoted by $-\PP$. In our paper $-\PP$ will have a very different meaning.}
We say that two polytopes $\PP$ and $\QQ$ are \emph{translation-equivalent} if there exists a vector $v\in V$ with $v+\PP=\QQ$. We denote by $\mathfrak{P}(V)$ the set of all translation-equivalence classes of polytopes in $V$. 

The \emph{Minkowski sum} of two polytopes $\PP$ and $\QQ$ in $V$
is defined as the polytope
\[ \PP+\QQ\,\,:=\,\,\{ p+q\,|\, p\in \PP\mbox{ and }q\in \QQ\}.\]
This turns $\mathfrak{P}(V)$ into an abelian monoid, where the identity element $0$ is represented by any  polytope consisting of a single point.
\begin{figure}[h]
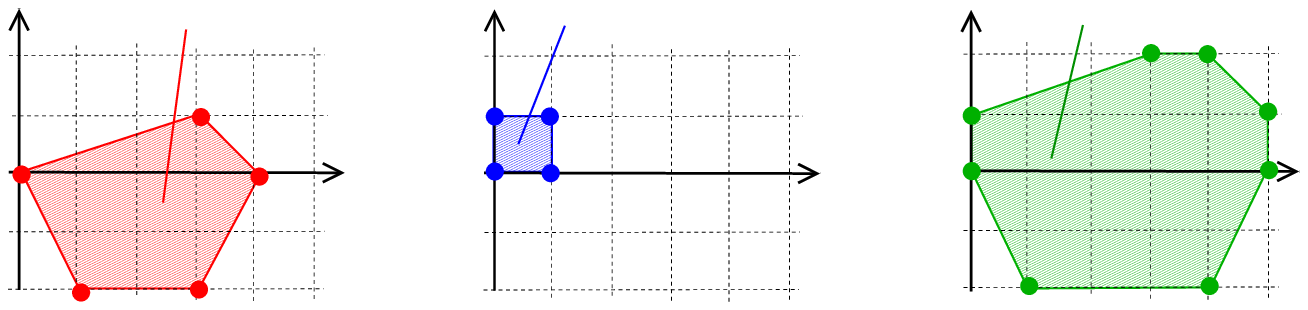
\caption{}\label{fig:minkowski-sum}
\end{figure}

It is straightforward to show, see e.g. \cite[Lemma~3.1.8]{Sc93},  that $\mfp(V)$ has the cancellation property, i.e.\ for any $\PP,\QQ,\RR\in \mfp(V)$ with $\PP+\QQ=\PP+\RR$ we have $\QQ=\RR$. 
We denote by $\mfg(V)$ the set of all equivalence classes of pairs $(\PP,\QQ)\in \mfp(V)^2,$ where we say that $(\PP,\QQ)\sim (\PP',\QQ')$ if $\PP+\QQ'=\PP'+\QQ$.
Note that $\mfg(V)$ is an abelian group, and since $\mfp(V)$ has the cancellation property it follows that the map  $\mfp(V)\to \mfg(V)$
given by $\PP\mapsto (\PP,0)$ is  a monomorphism. We will use this monomorphism to identify $\mfp(V)$ with its image in $\mfg(V)$. 
Given $\PP$ and $\QQ\in \mfp(V)$ we may write $\PP-\QQ=(\PP,\QQ)$. 
We refer to $\mfg(V)$ as the \emph{Grothendieck group of polytopes}.

Now let $\Gamma$ be a group. We write $\mfp(\Gamma)=\mfp(H_1(\Gamma;\R))$ and we say a polytope in $H_1(\Gamma;\R)$ is \emph{integral} if all the vertices lie in the image of $H_1(\Gamma;\Z)\to H_1(\Gamma;\R)$.
Similarly we write 
 $\mfg(\Gamma)=\mfg(H_1(\Gamma;\R))$ and we say an element in $\mfg(\Gamma)$ is \emph{integral} if it can be represented by the difference of two integral polytopes.\footnote{The invariants of a group $\pi$ we will define later on will lie in $\mfg(H_1(\pi;\R))$. In fact they will lie in a subgroup given by polytopes with integral vertices. Therefore it makes sense to study the polytope group of integral polytopes, which has also been studied in its own sake by Funke~\cite{Fu16}.} 
%ST: moved the footnote because integral was not defined in previous paragraph.
A group homomorphism $\varphi\colon \Gamma\to \Pi$ induces natural homomorphisms $\varphi_*\colon \mfp(\Gamma)\to \mfp(\Pi)$ and $\varphi_*\colon \mfg(\Gamma)\to \mfg(\Pi)$. 

%==================================================================
\subsection{$L^2$-acyclic groups and the Atiyah Conjecture}
\label{subsec:L2-acyclic_groups}
We say that a group $\Gamma$ is of \emph{type $F$} if it admits a finite model for $K(\Gamma,1)$.
(A group of type $F$ is well-known to be torsion-free, see e.g.\ \cite[Corollary~VIII.2.5]{Br82} for a proof.) We say $\Gamma$ is \emph{$L^2$-acyclic} if all  its $L^2$-Betti numbers $b_n^{(2)}(\Gamma)$ vanish. 
 The following gives examples of $L^2$-acyclic groups  of type $F$.
\bn
\item Fundamental groups of admissible 3-manifolds (here a $3$-manifold is \emph{admissible} if it is connected, orientable, irreducible,
  its boundary is empty or a disjoint union of tori, and its fundamental group
  is infinite).
\item Free-by-cyclic groups, i.e.\ groups that can be written as a semidirect product $F\rtimes \Z$ with $F$ a free group.\footnote{Strictly speaking these groups should be called ``free-by-infinite cyclic." But we follow the common convention to just say free-by-cyclic.}
\item Torsion-free groups with two generators and  one non-trivial relator.
\en
We refer to \cite[Theorem~0.1]{LoL95}, \cite[Section~4.2]{Lu02} and \cite{DL07} for proofs 
that these groups are indeed $L^2$-acyclic  of type $F$.

In this paper we are mostly  interested in groups for which the Whitehead group is trivial. It is conjectured that the Whitehead group is trivial for any torsion-free group. This conjecture has now been proven for large classes of groups. For example the Whitehead group is  known to be trivial for all torsion-free hyperbolic  groups~\cite{BLR08},  all virtually solvable~\cite{We15}  and  
the aforementioned  three types of groups, see e.g.\ \cite[(C.36)]{AFW15} and  \cite[p.~249~and~p.~250]{Wa78} for details.
We refer to \cite[Theorem~2]{KLR16} for a detailed list of all groups for which it is currently known that the Whitehead group is trivial.

Furthermore, for the most part we want to restrict ourselves to groups that satisfy the Atiyah Conjecture. A torsion-free group $\Gamma$ satisfies the Atiyah Conjecture if given any $(m\times n)$-matrix $A$ over $\Z[\Gamma]$ the $L^2$-dimension of the kernel of the map $r_A\colon l^2(\Gamma)^m\to l^2(\Gamma)^n$ defined by $v\mapsto v\cdot A$ is a natural number. The class of torsion-free groups that are known to satisfy the Atiyah Conjecture is considerably smaller than the class of torsion-free groups for which it is known that the Whitehead group is trivial.
We refer to \cite[Theorem~2.3]{LiL16} for a comprehensive summary of what is known about the Atiyah Conjecture. For us it is of interest that the Atiyah Conjecture is known  for  fundamental groups of most admissible 3-manifolds and for  free-by-cyclic groups, see e.g.\ \cite[(H.21)]{AFW15} and  \cite[Theorem~1.5]{Li93}.
In the following we refer to a group with trivial Whitehead group and which satisfies the Atiyah Conjecture as a \emph{$\whac$-group}. It is an open question whether  all torsion-free groups are $\whac$-groups.

%==================================================================
\subsection{The polytope invariant of $L^2$-acyclic groups  of type $F$}
In Section~\ref{section:definition-g-pi} we will use Reidemeister torsion over an appropriate skewfield to associate to an $L^2$-acyclic $\whac$-group $\Gamma$  of type $F$
an integral element $\PP(\Gamma)$ of $\mfg(\Gamma)$. We refer to $\PP(\Gamma)$ as the \emph{polytope invariant of $\Gamma$}. In Section~\ref{section:elementary-examples} we will see that there exists an $L^2$-acyclic $\whac$-group $\G$  of type $F$ such that neither $\PP(\G)$ nor $-\PP(\G)$ is represented by an actual polytope.  

The following theorem summarizes some of the structural properties of this invariant. We refer to Section~\ref{section:properties-of-p} for details. 

\begin{theorem}\label{thm:properties-of-p}
\bn[font=\normalfont]
\item Suppose $A,B$ and $C$ are $L^2$-acyclic $\whac$-groups  of type $F$. Let $\varphi\colon C\to A$ and $\phi\colon C\to B$ be monomorphisms. We consider the corresponding amalgamated  product $A*_CB$ and denote by $a,b$ and $c$ the monomorphisms from $A,B$ and $C$ into $A*_CB$. Then $A*_CB$ is also $L^2$-acyclic  of type $F$. Furthermore, if $A*_CB$ is a $\whac$-group, then
\[ \PP(A*_CB)\,\,=\,\, a_*(\PP(A))+b_*(\PP(B))-c_*(\PP(C)).\]
\item Let 
\[ 1\,\to\, K\,\xrightarrow{i_*}\, G\,\to\, B\to 1\] 
be an exact sequence of groups of type $F$. Suppose $G$ is a $\whac$-group and $K$ is $L^2$-acyclic. Then
\[ \PP(G)\,\,=\,\, i_*(\PP(K))\cdot \chi(B).\] 
\item 
 If $\Gamma$ is the fundamental group of an aspherical $n$-dimensional manifold $M$ and if  $\Gamma$ is an $L^2$-acyclic\footnote{The Singer Conjecture~\cite{Sin77}\cite[p.~421]{Lu02}  predicts that $M$ has zero $L^2$-Betti numbers if and only if $\chi(M)=0$.} $\whac$-group  of type $F$,  then
\[ \PP(\Gamma)\,\,=\,\,(-1)^{n+1}\cdot \ol{\PP(\Gamma)} \in \mfg(\Gamma).\]
\en
\end{theorem}

The polytope invariant is in general rather difficult to calculate. The following theorem summarizes what is known for special classes of groups.
In the interest of readability we keep the language somewhat informal. We refer to Section~\ref{section:examples} for more carefully formulated statements.

\begin{theorem}\label{thm:examples}
\bn[font=\normalfont]
\item If $N\not\cong S^1\times D^2$ is an admissible 3-manifold that is not a closed graph manifold, then $\PP(N)$ equals the dual of the Thurston norm ball. %~\cite{Th86}. ST: removed this reference as it looks as though the result is in Thurston's paper.
In particular $\PP(N)$ is an integral polytope and it determines and is determined by the Thurston norm.
\item If $\G$ is a group  that admits a presentation $\pi=\ll x,y|r\rr$ such that $r$ is non-empty, reduced and cyclically reduced, then $\PP(\G)$ can be easily read off the word $r$.
\item If $N$ is a closed admissible 3-manifold that is not a graph manifold and if $N$ admits a CW-structure with one 0-cell, two 1-cells, two 2-cells and one 3-cell, then $\PP(N)$ can be determined immediately from the corresponding chain complex of the universal cover.
\en
\end{theorem}

It follows in particular from Theorem~\ref{thm:examples} that for 3-manifolds to which (2) or (3) applies one can easily obtain the Thurston norm from the fundamental group or the chain complex. 

Next we turn to the question of what information the polytope invariant contains. The following theorem gives a  partial answer to that question. It is again formulated in a slightly informal way, the precise statements will be given later 
in Sections~\ref{section:marked-polytopes} and~\ref{section:polytope-meaning}.

\begin{theorem}\label{thm:polytope-meaning}
\bn[font=\normalfont]
\item If $N\ne S^1\times D^2$ is an admissible 3-manifold that is not a closed graph manifold, then we can mark some of the vertices of  the polytope $\PP(\pi_1(N))$ such that a class $\phi\in H^1(\pi_1(N);\R)$ pairs maximally with a marked vertex if and only if it represents an element in the BNS-invariant~\cite{BNS87} $\Sigma(\pi_1(N))$.
\item Let $\G$ be a group with $b_1(\G)=2$ that admits a presentation $\pi=\ll x,y|r\rr$ such that $r$ is non-empty, reduced and cyclically reduced. Then the following hold:
\bn[font=\normalfont]
\item The ``thickness'' of $\PP(\G)$ is determined and determines the minimal complexity of ``HNN-splittings'' of $\G$ along groups.
\item We can mark some of the vertices of the polytope  $\PP(\G)$ such that a class $\phi\in H^1(\G;\R)$ pairs maximally with a marked vertex if and only if it represents an element in the BNS-invariant $\Sigma(\Gamma)$.
\en
\item \emph{[Funke--Kielak]} If $\Gamma$ is an descending HNN-extension of a free group on two generators, then the geometry of $\PP(\Gamma)$ is closely related to $\Sigma(\Gamma)$.
\en
\end{theorem}

Most of the results in this paper are already explicit or at least implicit in our previous papers~\cite{FT15,FST15,FL16a,FL16b}. Only  Theorem~\ref{thm:examples}~(3) is a completely new statement. \\
%Theorem~\ref{thm:3-manifold-with-two-generators}

The paper is organized as follows.
In Section~\ref{section:definition-g-pi} we will use Reidemeister torsion to introduce the polytope invariant $\PP(\G)$ for any $L^2$-acyclic $\whac$-group $\Gamma$  of type $F$. Furthermore we prove several statements regarding polytopes and Reidemeister torsions that will in particular imply  
Theorem~\ref{thm:properties-of-p}. In Section~\ref{section:examples} we give a more detailed discussion of the various statements of Theorem~\ref{thm:examples}.
In Sections~\ref{section:marked-polytopes} and~\ref{section:polytope-meaning} we will explain in more detail the statement and the references for Theorem~\ref{thm:polytope-meaning}.
We conclude this paper with a long list of questions in Section~\ref{section:questions}.

\subsection*{Acknowledgments.} 
We wish to thank the many people who shared their ideas and insight with us as we developed the subject. We are especially grateful to Nathan Dunfield, Florian Funke, Dawid Kielak and Kevin Schreve for their contributions. 

The first author is supported by the SFB 1085 ``Higher invariants'' at the University of Regensburg, funded by the Deutsche Forschungsgemeinschaft (DFG). The second author is 
 financially supported by the Leibniz-Preis granted by the {DFG} and the ERC Advanced Grant ``KL2MG-interactions''
(no.  662400)  granted by the European Research Council.
 The third author is partially supported under the Australian Research Council's Discovery funding scheme (project number DP140100158).

%==================================================================
\section{Definition of the polytope invariant of groups}\label{section:definition-g-pi}

%===============================================
\subsection{Review of division and rational closure}
\label{subsec:Review_of_division_closure}
Let $R$ be a subring of a ring $S$. The \emph{division closure} $\mathcal{D}(R \subseteq S) \subseteq S$ 
is the smallest subring of $S$ which contains $R$ and is division closed,
i.e., any element $x \in \mathcal{D}(R \subseteq S)$ which is invertible in $S$ is already invertible in
$\mathcal{D}(R \subseteq S)$. The \emph{rational closure} $\mathcal{R}(R \subseteq S) \subseteq S$ is
the smallest subring of $S$ which contains $R$ and is rationally closed, i.e., for any
natural number $n$ and matrix $A \in M_{n,n}(\mathcal{R}(R \subseteq S)),$ if $A$ is invertible over
$S,$ then $A$ is already invertible over $\mathcal{R}(R \subseteq S)$. 
The division closure and the
rational closure always exist. Obviously 
$R \subseteq \mathcal{D}(R \subseteq S) \subseteq \mathcal{R}(R \subseteq S) \subseteq S$.

Consider a group $\Gamma$. Let $\mathcal{N}(\Gamma)$ be the group von Neumann algebra which can be
identified with the algebra $\mathcal{B}(L^2(\Gamma),L^2(\Gamma))^\Gamma$ of bounded left $\Gamma$-equivariant operators
$L^2(\Gamma) \to L^2(\Gamma)$.  Denote by $\mathcal{U}(\Gamma)$ the algebra of operators that are affiliated
to the group von Neumann algebra, see \cite[Section~8]{Lu02} for details.  This is the same as the Ore localization of $\mathcal{N}(\Gamma)$
with respect to the multiplicatively closed subset of non-zero divisors in $\mathcal{N}(\Gamma)$,
see~\cite[Theorem~8.22~(1)]{Lu02}.  By the right regular representation we can embed 
$\C \Gamma$ and hence also $\Z \Gamma$ as a subring in $\mathcal{N}(\Gamma)$.  We will denote by $\mathcal{R}(\Gamma)$ and
$\mathcal{D}(\Gamma)$ the division and the rational closure of $\Z \Gamma$ in $\mathcal{U}(\Gamma)$.  Summarizing we get a
commutative diagram of inclusions of rings
\[
\xymatrix@!C= 8em@R0.5cm{
\Z \Gamma \ar[r] \ar[d]
&
\mathcal{N}(\Gamma) \ar[dd]
\\
\mathcal{D}(\Gamma) \ar[d]
&
\\
\mathcal{R}(\Gamma) \ar[r]
& \mathcal{U}(\Gamma).}
\]
We will use these inclusions to identify each ring with its monomorphic images.

The following lemma is well-known to the experts, full references are given in \cite[Lemma~1.21]{FL16b}.

\begin{lemma} \label{lem:L2-acyclic-and-Zpi-contractible}
Let $C_*$ be a finite  chain complex of  free left-$\Z \Gamma$-modules.  Then the following assertions are equivalent:

\begin{enumerate}[font=\normalfont]
\item The $L^2$-Betti numbers of $C_*$ are all zero.
\item
The $\mathcal{R}(\Gamma)$-chain complex $\mathcal{R}(\Gamma) \otimes_{\Z \Gamma} C_*$ is contractible.
\end{enumerate}
\end{lemma}

\begin{lemma}\label{lem:rational-equals-division-closure}
Let $\Gamma$ be a group. If $\Gamma$ is torsion-free and  if $\Gamma$ satisfies the Atiyah Conjecture, then the rational closure $\mathcal{R}(\Gamma)$ agrees with the division closure $\mathcal{D}(\Gamma)$ of $\Z \Gamma
\subseteq \mathcal{U}(\Gamma)$. Furthermore, $\mathcal{D}(\Gamma)$ is a skew field. 
\end{lemma}

This lemma is proved in~\cite[Lemma~10.39]{Lu02} 
for the Atiyah Conjecture over $\C$ with the division closure of $\C[\Gamma]$ instead of the division closure of $\Z[\Gamma]$, but the proof verbatim implies Lemma~\ref{lem:rational-equals-division-closure}.

%%%%%%%%%%%%%%%%%%%%%%%%%%%%%%%%%%%%%%%%%%%%%%%%%%%%%%%%%%%%%%%%%%%%%%%%%%%%%%%%%
\subsection{The polytope homomorphism}
\label{section:The_polytope_homomorphism}
Throughout this section let $\Gamma$ be a torsion-free group that  satisfies the Atiyah Conjecture. By Lemma~\ref{lem:rational-equals-division-closure} $\DD(\Gamma)$ is a skew field.
Our goal is to associate to any element in the multiplicative group of units $\DD(\Gamma)^\times=\DD(\Gamma)\sm \{0\}$ an integral element of $\mfg(\Gamma)$. We outline the main steps in the construction and refer to \cite{FL16a}, where 
this map was first introduced, for more details.

Write $H:=H_1(\Gamma;\Z)/\mbox{torsion}$ and view $H$ as a multiplicative group. Let $\op{pr}\colon \Gamma\to H$ be the canonical projection and $K$ be the kernel of $\op{pr}$. Choose a map of sets 
$s \colon H \to \Gamma$ with $\op{pr} \circ s = \id_{H}$.
We denote by $\Z[K]*_s H$ the ring, often referred to as the crossed product of $\Z[K]$ and $H$,  whose underlying abelian group is given by finite formal sums $\sum_{h\in H} a_h h$ with each $a_h\in \Z[K]$ and for which multiplication is  extended from the multiplication on $\Z[K]$ by the rule $g\cdot h=s(g)\cdot s(h)\cdot s((gh)^{-1})\cdot gh$ for $g,h\in H$
and by the rule $h\cdot k=  (hk    s(h)^{-1})\cdot  h$ for $h\in H$ and $k\in K$. 
It is straightforward to see that  $\Z[K]*_s H$ is a ring and that 
\[ \ba{rcl} \Z[\Gamma] &\to & \Z[K]*_s H\\
\tmsum{g\in \G}{} a_g g&\mapsto & 
\tmsum{g\in \G}{} a_g gs(\op{pr}(g))^{-1}\cdot \op{pr}(g)
\ea
\]
is a ring isomorphism, which we will sometimes use to identify these two rings.
 Given a non-zero element $\sum_{h\in H} a_h h\in \Z[K]*_s H$ we consider the integral polytope
\[ \PP\Big(\tmsum{h\in H}{} a_h h\Big)\,=\, \mbox{(convex hull of $h$ with $a_h\ne 0$)}\,\,\in \,\, \PP(H)\,=\,\PP(\Gamma).\]
Defining $\DD(K)*_s H$ analogously, given any non-zero element $f\in \DD(K)*_s H$ we obtain the corresponding polytope $\PP(f)$. 
Since $\DD(K)$ is a skew field, in particular a domain, we obtain from elementary arguments that for any non-zero $f,g\in \DD(K)*_s H$ we have
\be \label{equ:p-multiplicative} \PP(f\cdot g)\,=\, \PP(f)+\PP(g).\ee

We denote by $T=(\DD(K)*_s H)\sm \{0\}$ the set of all non-zero elements of the domain $\DD(K)*_s H$. We can form the Ore localization $T^{-1}(\DD(K)*_s H)$---a proof of this fact is for example given in  \cite[Theorem~6.4]{DLMSY03} or \cite[Example~8.16]{Lu02}. There exists a canonical isomorphism 
\[
T^{-1} (\mathcal{D}(K) \ast_s H)\,\, \xrightarrow{\cong} \,\,\mathcal{D}(\Gamma),
\]
which we will use to identify these rings, see e.g.\ \cite[Lemma~10.69]{Lu02}. (Again the lemma in \cite{Lu02} is stated for the division closure of  $\C[G]$ but it also holds with the same proof for the division closure of $\Z[G]$.)
Now let $h=fg^{-1}\in \DD(\Gamma)=T^{-1} (\mathcal{D}(K) \ast_s H)$ be non-zero. We define
\[ \PP(h)\,\,:=\,\, \PP(f)-\PP(g)\,\in \,\mfg(\Gamma).\]
It follows from (\ref{equ:p-multiplicative}) that this is well-defined and that this map defines a group homomorphism $\PP\colon \DD(\Gamma)^\times \to \mfg(\Gamma)$.
Since the target is abelian this descends to a group homomorphism 
\[ \PP\colon \DD(\Gamma)^\times/[ \DD(\Gamma)^\times,\DD(\Gamma)^\times]\,\, \to\,\, \mfg(\Gamma).\]
One easily checks that this homomorphism is independent of the choice of $s$.

%%%%%%%%%%%%%%%%%%%%%%%%%%%%%%%%%%%%%%%%%%%%%%%%%%%%%%%%%%%%%%%%%%%%%%%%%%%%%%%%
\subsection{The polytope invariant of an $L^2$-acyclic group  of type $F$}
\label{section:properties-of-p}\label{section:definition-of-p(gamma)}
Given a ring $R$ we denote by $K_1(R)$ the usual $K_1$-group, as defined in \cite{Sil81,Ro94}. If $\K$ is a skew field, then the Dieudonn\'e determinant, see~\cite[Corollary~4.3 on page~133]{Sil81} and \cite{Ro94}, gives rise to an isomorphism
\[
{\det} \colon  K_1(\K)
\,\, \xrightarrow{\cong}  \,\,
\K^\times_{\op{ab}}:=\K^{\times}/[\K^{\times},\K^{\times}]. 
\]
In the following, given a torsion-free group $\Gamma$ that satisfies the Atiyah Conjecture we will use this canonical isomorphism to identify $K_1(\DD(\Gamma))$ with $\DD(\Gamma)^\times_{\op{ab}}$. 

\begin{definition}
\bn
\item An \emph{$L^2$-acyclic pair} $(X,\varphi)$ consists of a finite connected CW-complex and a homomorphism  $\varphi\colon\pi_1(X)\to \Gamma$ such that $b_i^{(2)}(X,\varphi)=0$  for all $i$. Here $b_i^{(2)}(X,\varphi)$ denotes the $L^2$-Betti numbers of the covering space of $X$ corresponding to $\varphi$, viewed as a $\Gamma$-CW-complex.
\item Suppose that \emph{$(X,\varphi\colon \pi_1(X)\to \Gamma)$ is an $L^2$-acyclic pair} and  suppose that $\Gamma$ is torsion-free and that it satisfies the Atiyah Conjecture. 
We denote by $\wti{X}$ the universal cover of $X$. By picking orientations of the cells of $X$ and by picking lifts of the cells of $X$ to $\wti{X}$ we can view 
$C_*(\wti{X})$ as a chain complex of based  free $\Z[\pi_1(X)]$-left modules.
Similarly we can view  $\DD(\Gamma)\otimes_{\Z[\Gamma]} C_*(\wti{X})$ as a chain complex of based free $\DD(\Gamma)$-left modules. 
We had assumed  that $(X,\varphi)$ is $L^2$-acyclic. Together with Lemmas~\ref{lem:L2-acyclic-and-Zpi-contractible} and~\ref{lem:rational-equals-division-closure} this implies that the chain complex 
$\DD(\Gamma)\otimes_{\Z[\Gamma]} C_*(\wti{X})$ is acyclic.
We denote by $\tau(X,\varphi)\in K_1(\DD(\Gamma))/\pm \Gamma=\DD(\Gamma)^\times_{\op{ab}}/\pm \Gamma$ the corresponding torsion as defined in \cite{Mi66}.
Furthermore we define the \emph{polytope invariant of $(X,\varphi)$} as 
\[\PP(X,\varphi)\,:=\,-\PP(\tau(X,\varphi))\,\in\, \mfg(\Gamma).\]
\en
\end{definition}

The minus sign in the definition of $\PP(X,\varphi)$ might initially be a little surprising. This choice of sign ensures that in many situations of interest, the polytope invariant $\PP(X,\varphi)$ lies in $\mfp(\Gamma)$, i.e.\ it indeed can be represented by a polytope.

The following proposition summarizes some of the key properties of the polytope invariant of an $L^2$-acyclic pair.
The proof of the proposition follows from standard properties of Reidemeister torsion and the fact that the map $\PP\colon K_1(\DD(\Gamma))\to \mfg(\Gamma)$ is a homomorphism. We refer to \cite[Theorem~2.5]{FL16b} for details.

\begin{proposition}\label{prop:properties-of-p}
Let $(X,\varphi\colon \pi_1(X)\to \Gamma)$ be an $L^2$-acyclic pair and suppose that  $\Gamma$ is a torsion-free group that satisfies the Atiyah Conjecture.
\bn[font=\normalfont]
\item \emph{(Induction)} If $\delta \colon \Gamma\to \Pi$ is a monomorphism to a torsion-free group $\Pi$ that satisfies the Atiyah Conjecture, then $\PP(X,\delta\circ \varphi)=\delta_*(\PP(X,\varphi))$, where $\delta_*\colon \mfg(\Gamma)\to \mfg(\Pi)$ is the induced map.
\item \emph{(Simple homotopy invariance)} If $f\colon W\to X$ is a simple homotopy equivalence of CW-complexes, i.e.\ if $\Wh(f)=0\in \Wh(\pi_1(X))$, then $\PP(W,\varphi\circ f_*)=\PP(X,\varphi)$.
\item \emph{(Homeomorphism invariance)} If $f\colon W\to X$ is a homeomorphism, then we have $\PP(W,\varphi\circ f_*)=\PP(X,\varphi)$.
\en
\end{proposition}

Now we can make the following two definitions.
\bn
\item  We say that a group $\Gamma$ is \emph{$L^2$-acyclic  of type $F$} if it admits a finite $K(\Gamma,1)$ and if all its $L^2$-Betti numbers vanish. If $\Gamma$ is a $\whac$-group, then  we define \[ \PP(\Gamma)\,\,:=\,\,\PP(X,\id)\]
where $X$ is any finite $K(\Gamma,1)$.
It follows from 
Proposition~\ref{prop:properties-of-p} (2) that this definition does not depend on the choice of $X$. We refer to $\PP(\Gamma)$ as the \emph{polytope invariant} of $\Gamma$. 
\item
Let $M$ be a compact manifold and let  $\varphi\colon \pi_1(M)\to \Gamma$ be a homomorphism to a group  $\Gamma$ that satisfies the Atiyah Conjecture. As above we say that $(M,\varphi)$ is $L^2$-acyclic if $b_n^{(2)}(M,\varphi)=0$ for all $n\in \N_0$. 
Now suppose that $(M,\varphi)$ is $L^2$-acyclic. 
We pick a CW-structure $X$ for $M$ and we define $\PP(M,\varphi):=\PP(X,\varphi)$. It follows from 
Proposition~\ref{prop:properties-of-p} (3) that this definition does not depend on the choice of $X$. Sometimes we write $\PP(M):=\PP(M,\id)$. 
\en

The following proposition collects a few more structural properties of the polytope invariant.
These are again a consequence of \cite[Theorem~2.5]{FL16b}.

\begin{proposition}\label{prop:properties-of-p-2}
\bn[font=\normalfont]
\item Let $X=A\cup_C B$ be a decomposition of a finite CW-complex $X$ into two connected CW-complexes $A$ and $B$ such that $C:=A\cap B$ is also connected.
Let $\varphi\colon \pi_1(X)\to \Gamma$ be a homomorphism to a torsion-free group $\Gamma$ that satisfies the Atiyah Conjecture.
 We denote by $a\colon A\to  X$, $b\colon B\to X$ and $c\colon C\to X$ the inclusion maps. If $(A,\varphi\circ a_*)$, 
$(B,\varphi\circ b_*)$ and 
$(C,\varphi\circ c_*)$ are $L^2$-acyclic, then $(X,\varphi)$ is also $L^2$-acyclic and 
\[ \PP(X,\varphi)\,\,=\,\,a_*(\PP(A,\varphi\circ a_*))
+b_*(\PP(B,\varphi\circ b_*))-c_*(\PP(C,\varphi\circ c_*)).\]
\item Let 
\[ 1\,\to\, F\,\xrightarrow{i}\, E\,\to\, B\,\to\,1\] 
be a fibration of finite CW-complexes.
Let $\varphi\colon \pi_1(E)\to \Gamma$ be a homomorphism to a torsion-free group $\Gamma$ that satisfies the Atiyah Conjecture. If $(F,\varphi\circ i_*)$ is $L^2$-acyclic, then
\[ \PP(E,\varphi)\,\,=\,\, \chi(B)\cdot i_*(\PP(F,\varphi\circ i_*)).\]
\item Let $M$ be an $n$-dimensional closed orientable manifold and let $\varphi\colon \pi_1(M)\to \Gamma$ be a homomorphism 
to a group $\Gamma$ that satisfies the Atiyah Conjecture. If $(M,\varphi)$ is $L^2$-acyclic, then 
\[ \PP(M,\varphi)\,\,=\,\,(-1)^{n+1}\cdot \ol{\PP(M,\varphi)}\,\in\, \mfg(\Gamma).\]
\en
\end{proposition}

It is clear that Theorem~\ref{thm:properties-of-p} is an immediate consequence of Propositions~\ref{prop:properties-of-p} and~\ref{prop:properties-of-p-2} and the definitions.

%==================================================================
\section{Examples}\label{section:examples}

%==================================================================
\subsection{Elementary examples}\label{section:elementary-examples}
We first consider the infinite cyclic group $\ll t\rr$. 
In this case we take $X=K(\ll t\rr,1)=S^1$. The cellular chain complex of the universal cover of $\R$ is then isomorphic to \[0\,\,\to\,\, \zt\xrightarrow{\cdot (1-t)}\zt\,\,\to\,\, 0.\]
 It follows that  $\tau(X,\id)=(1-t)^{-1}$ and therefore $\PP(\ll t\rr)=-[0,1]\in \mfg(\ll t\rr)=\mfg(\R)$. In this case $\PP(\ll t\rr)$ is therefore not represented by a polytope.

Now consider $A=\ll s\rr \times F_3$, where $F_3$ denotes, as usual, the free group on three generators. It follows from the above calculation and from Theorem~\ref{thm:properties-of-p}~(2) that $\PP(A)$ is represented by an interval of length $-\chi(F_3)=2$. We also consider $B=\Gamma=\ll t\rr \times (F_3\times F_5)$. (Strictly speaking we do not know of a proof that $B$ satisfies the Atiyah Conjecture, the following discussion implicitly assumes that this is the case so that we can define $\PP(B)$.) Similar to the above we see that $\PP(B)$ is represented by minus an interval of length $\chi(F_3)\cdot \chi(F_5)=8$.
Now let $C=\Z$ and pick a monomorphism $\varphi\colon C\to A$ that factors through a monomorphism $C\to F_3$ and similarly pick a monomorphism $\psi\colon C\to B$ that factors through a monomorphism $C\to F_3\times F_5$. We denote by $\Gamma$ the corresponding  amalgamated product. It is a consequence of  Theorem~\ref{thm:properties-of-p}~(2) that $\PP(\Gamma)$ is a difference of two intervals and it is straightforward to see that $\PP(\Gamma)$ is neither represented by a polytope nor by minus a polytope.

%==================================================================
\subsection{Groups with two generators and one relator}
In general it is very hard to compute the polytope invariant, the main difficulty lies in determining the Dieudonn\'e determinant of a square matrix. There is only one situation in which the calculation of the Dieudonn\'e determinant is straightforward, namely when the matrix is a $1\times 1$-matrix. As we will see, this observation makes it straightforward to determine the polytope invariant for 
groups with two generators and one relator.

We say that a presentation $\pi=\ll x,y|r\rr$ is \emph{nice} if it satisfies the following conditions:
\bn
\item $r$ is a non-empty, reduced and cyclically reduced word, and
\item $b_1(\G_\pi)=2$, where $\G_\pi$ denotes the group defined by the presentation $\pi$.
\en

Following \cite{FT15} we will associate to a nice presentation $\pi=\ll x,y|r\rr$ a polytope $\SS(\pi)$. The definition is illustrated in Figure~\ref{fig:brown-intro}.   First we identify $H_1(\G_\pi;\Z)$ with $\Z^2$ such that $x$ corresponds to $(1,0)$ and $y$ corresponds to $(0,1)$. 
Then the relator $r$ determines a discrete walk on the integer lattice in $H_1(\G_\pi;\R)=\Z^2$ and the  polytope $\SS(\pi)$ is obtained from the convex hull of the trace of this walk in the following way:
\bn
\item Start at the origin and walk across $\Z^2$ reading the word $r$ from the left.
\item Take the convex hull $\mathcal{C}$ of the set of all lattice points reached by the walk.
\item An elementary argument, using the fact that $r$ is reduced and cyclically reduced, shows that one can take the Minkowski difference with the square $\mathcal{Q}$ of length one, i.e.\ there exists a  polytope   $\SS(\pi)$ with $\SS(\pi)+\mathcal{Q} =\mathcal{C}$.
\en
Figure~\ref{fig:brown-intro} illustrates the construction of the  polytope for the nice presentation
$$\pi=\ll x,y\,|\, yx^4yx^{-1}y^{-1}x^2y^{-1}x^{-2}y^2xy^{-1}xy^{-1}x^{-1}y^{-2}x^{-3}y^2x^{-1}\rr.$$
For the final Minkowski difference see also Figure~\ref{fig:minkowski-sum}.

\begin{figure}[h]
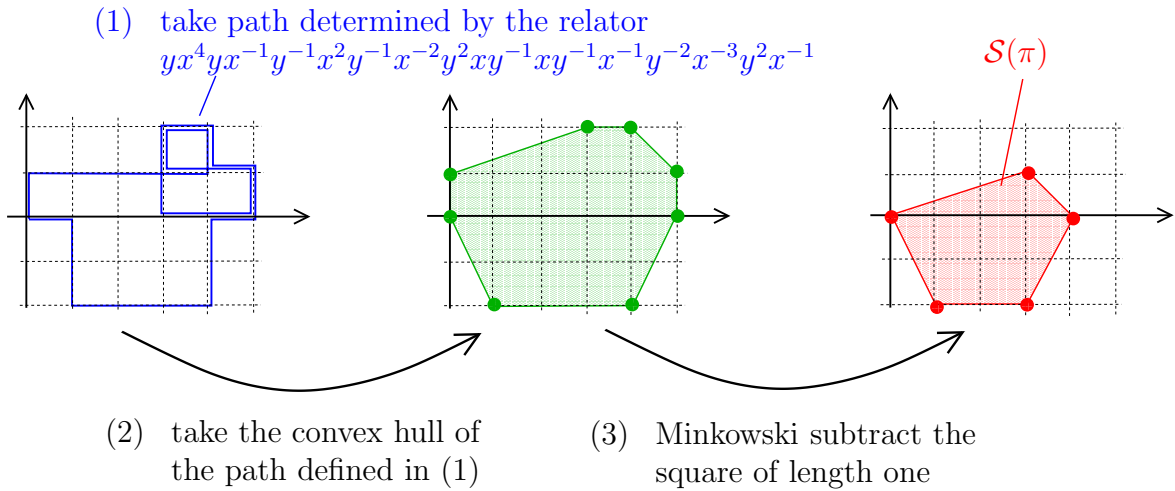
\caption{The polytope $\SS(\pi)$ for a presentation $\pi$}
\label{fig:brown-intro}
\end{figure}

We now view $\SS(\pi)$ as an element in $\mfp(\Z^2)=\mfp(\Gamma_\pi)$, where we identify $\Z^2$ with $H_1(\Gamma_\pi;\Z)$ by sending $(0,1)$ to $x$ and $(0,1)$ to $y$. A priori $\SS(\pi)\in \mfp(\G_\pi)$ is an invariant of the presentation and not of the underlying group $\Gamma_\pi$.

The following theorem gives another strong indication that the polytope is an invariant of the group.

\begin{theorem}\label{thm:polytopes-for-2-1-presentations}
Let $\pi=\ll x,y\,|\, r\rr$  be a nice presentation. If $\Gamma_\pi$ is torsion-free\footnote{The group $\G_\pi$ is torsion-free if and only if the word $r$ is not a proper power of another word, see e.g.\ \cite[Proposition~II.5.17]{LS77} for details.} and if $\Gamma_\pi$ is a $\whac$-group, then $\Gamma_\pi$ is $L^2$-acyclic   of type $F$ and 
\[ \SS(\pi)\,\,=\,\, \PP(\G_\pi)\,\in\,\mfp(\G_\pi).\]
\end{theorem}

\begin{remark}
In \cite[Theorem~1.3]{FT15} it was shown  $\SS(\pi)$ is an invariant of the underlying group if $\Gamma_\pi$ is residually a group that is elementary-amenable and torsion-free (without the assumption that $\Gamma_\pi$ satisfies the Atiyah Conjecture). In \cite[Conjecture~1.9]{Wi14} Wise (see also \cite[p.~2]{BK15}) conjectured that most hyperbolic groups $\ll x,y|r\rr$   act properly and cocompactly on a $\mbox{CAT}(0)$ cube complex. By Agol's Theorem \cite{Ag13} a proof of this conjecture would imply that such groups are virtually special. By \cite[Theorem~3.8]{FST15} together with \cite[(H.26)]{AFW15} this would imply that any word hyperbolic torsion-free group with a nice presentation $\ll x,y|r\rr$ is residually a group
that is elementary-amenable and torsion-free. Summarizing, a proof of Wise's Conjecture would show that $\SS(\pi)$ is an invariant of the underlying group for torsion-free hyperbolic groups with a nice $(2,1)$-presentation.
\end{remark}

\begin{proof}[Proof of Theorem~\ref{thm:polytopes-for-2-1-presentations}]
By \cite[Proposition~III.11.1]{LS77} the 2-complex $X$ corresponding to the given presentation is aspherical. We denote the universal of $X$ by $\wti{X}$ and we write $\G=\G_\pi$. As usual we view $\Z[\G]$ as a subring of $\DD(\G)$. The chain complex $\DD(\G)\otimes_{\Z[\Gamma]} C_*(\wti{X})$ is isomorphic to
\[ 0\,\to\, \DD(\G)\xrightarrow{{\bp r_x&r_y\ep}}\DD(\G)^2\xrightarrow{\footnotesize{\bp x-1\\ y-1\ep}}\DD(\G)\,\to\,0\]
where $r_x:=\frac{\partial r}{\partial x}$ and $r_x:=\frac{\partial r}{\partial x}$ denote the Fox derivatives~\cite{Fo53} of the word $r$ with respect to the generators $x$ and $y$.
By \cite{DL07} and  Lemmas~\ref{lem:L2-acyclic-and-Zpi-contractible} and~\ref{lem:rational-equals-division-closure} this chain complex is acyclic. Since $y-1\ne 0\in \DD(\G)$ it follows from standard arguments, see e.g.\ \cite[Theorem~2.2]{Tu01} or \cite[Lemma~3.1]{DFL16} that the torsion of the chain complex equals $r_x\cdot (y-1)^{-1}\in \DD(\G)$. So we obtain
\[ \PP(\G)\,\,=\,\,\PP\big(\tau\big(\DD(\G)\otimes_{\Z[\Gamma]} C_*(\wti{X})\big)\big)\,\,=\,\,\PP(r_x\cdot (y-1)^{-1})\,\,=\,\,\PP(r_x)-\PP(y-1).\]
But by \cite[Proposition~3.5]{FT15} the polytope $\SS(\pi)$ agrees with  $\PP(r_x)-\PP(y-1)$ in $\mfp(\Gamma_\pi)$.
\end{proof}

In \cite[Section~8]{FT15} we also deal with presentations of the form $\ll x,y|r\rr,$ where $r$ is non-empty, reduced and cyclically reduced, but where $b_1(\Gamma_\pi)=1$. In this case we can also `naively' define a polytope  (i.e.\ an interval) $\SS(\pi)$ in $H_1(\Gamma_\pi;\R)=\R$  and the obvious analogue of  Theorem~\ref{thm:polytopes-for-2-1-presentations} holds.

%==================================================================
\subsection{$3$-manifolds}
Let $N$ be a 3--manifold. For each $\phi\in H^1(N;\mathbb{Z})$ there is a properly embedded oriented surface $\Sigma$, such that 
$[\Sigma]\in H_2(N, \partial N; \mathbb{Z})$ is the Poincar\'e dual to $\phi$. Letting 
$\chi_-(\Sigma)=\sum_{i=1}^k \max\{-\chi(\Sigma_i),0\}$, 
where $\Sigma_1, \ldots, \Sigma_k$ are the connected components of~$\Sigma$, the \emph{Thurston norm} of $\phi\in H^1(N;\mathbb{Z})$  is  defined as 
 \[
x_N(\phi)\,:=\,\min \big\{ \chi_-(\Sigma) \,|\, \mbox{$\Sigma$ is a properly embedded surface with $\op{PD}([\Sigma]) = \phi$} \big\}.
\]
Thurston \cite{Th86} showed that $x_N$ is a seminorm on $H^1(N;\Z)$ and elementary arguments show that $x_N$ extends to a seminorm $x_N$ on $H^1(N;\R)$. We denote by 
\[ \TT(N)\,:=\,\{v\in H_1(N;\R)\,|\, \phi(v)\leq 1\mbox{ for all }\phi\in H^1(N;\R)\mbox{ with }x_N(\phi)\leq 1\}\]
the dual of the unit norm ball of $x_N$. Thurston~\cite{Th86} showed that $\TT(N)$ is a polytope with integral vertices. 

The following theorem is \cite[Theorem~3.35]{FL16b}. The proof relies on the recent work of Agol~\cite{Ag08,Ag13}, Przytycki-Wise~\cite{PW12} and Wise~\cite{Wi09,Wi12a,Wi12b}. 

\begin{theorem}\label{thm:polytopes-agree-for-3-manifold}
For any admissible 3-manifold $N\ne S^1\times D^2$ that is not a closed graph manifold we have
\[ \TT(N)\,\,=\,\, 2\cdot \PP(\pi_1(N))\,\,\in \,\mfp(\pi_1(N)).\]
\end{theorem}

The combination of Theorems~\ref{thm:polytopes-for-2-1-presentations}
and~\ref{thm:polytopes-agree-for-3-manifold} gives us the following corollary.
This corollary was first proved in \cite{FST15} using a different approach.

\begin{corollary}\label{cor:thurstonnorm-2-1}
Let $N$ be an admissible 3-manifold such that $\pi_1(N)$ admits a nice presentation $\pi=\ll x,y|r\rr$. Then
\[ \TT(N)\,\,=\,\,2\cdot \SS(\pi).\]
\end{corollary}

%==================================================================
\subsection{Closed 3-manifolds with a two-generator fundamental group}
The following theorem says that for ``small'' closed 3-manifolds one can easily obtain the Thurston norm from the chain complex of the universal cover. This theorem can be viewed as a version of Corollary~\ref{cor:thurstonnorm-2-1} for closed 3-manifolds. It applies, in particular, to all closed admissible 3--manifolds of Heegaard genus equal to two. It is known through work of Kobayashi~\cite{Ko88} and Hempel~\cite{He01} in combination with Perelman's solution of the geometrisation conjecture that if the splitting distance of a genus two Heegaard splitting of a 3-manifold is larger than two, then the 3-manifold is hyperbolic (whence admissible). Maher~\cite{Ma10} used this fact to show that \emph{most} 3--manifolds of Heegaard genus two are hyperbolic. 

\begin{theorem}\label{thm:3-manifold-with-two-generators}
Let $N$ be a closed admissible 3-manifold that is not a  graph manifold. We write $\pi=\pi_1(N)$. Suppose $N$ admits a CW-structure with one 0-cell, two 1-cells, two 2-cells and one 3-cell. The corresponding cellular chain complex of the universal cover $\wti{N}$ is of the form
\[ 0\to \Z[\pi]\xrightarrow{\scriptsize{\bp c_{1}&c_2\ep}} \Z[\pi]^2\xrightarrow{\scriptsize{\bp b_{11}&b_{12}\\ b_{21}&b_{22}\ep}} \Z[\pi]^2\xrightarrow{\scriptsize{\bp a_1\\ a_2\ep}}\Z[\pi]\to 0.\]
Then there exist $i,j\in \{1,2\}$ with $c_i\ne 0$ and $a_j\ne 0$, and 
\[ \TT(N)\,\,=\,\, \PP(b_{3-i,3-j})-\PP(c_i)-\PP(a_j).\]
\end{theorem}

\begin{proof}
As we pointed out earlier, the chain complex $\DD(\pi)\otimes_{\Z[\pi]} C_*(\wti{N})$ is acyclic. It thus follows that there exist $i,j\in \{1,2\}$ with $c_i\ne 0$ and $a_j\ne 0$. Since $\DD(\pi)$ is a skew field one can show, similar to \cite[Theorem~2.2]{Tu01}, that 
\[ \tau(N)\,\,=\,\, c_i^{-1}\cdot b_{3-i,3-j}\cdot a_j^{-1}.\]
This implies that $\PP(\pi_1(N))=\PP(b_{3-i,3-j})-\PP(c_i)-\PP(a_j)$. The theorem now follows from Theorem~\ref{thm:polytopes-agree-for-3-manifold}.
\end{proof}

%==================================================================
\section{Marked polytopes}\label{section:marked-polytopes}
A \emph{marked polytope} is a pair $\MM=(\PP,\VV)$, where $\PP$ is a polytope and $\VV$ is a (possibly empty) subset of the set of vertices of $\PP$. We refer to the vertices in $\VV$ as the \emph{marked vertices}.
If $\MM$ and $\NN$  are two marked polytopes, then we define the \emph{$($marked$)$ Minkowski sum $\MM+\NN$ of $\MM$ and $\NN$} as the Minkowski sum of the corresponding polytopes, where the marked vertices of the Minkowski sum are  precisely those that are the sum of a marked vertex of $\MM$ and a marked vertex of $\NN$. An example is given in Figure~\ref{fig:summp}, where the marked vertices are indicated by a dot.

\begin{figure}[h]
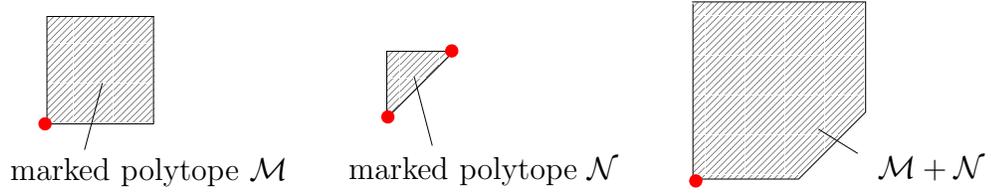
\caption{Example of the Minkowski  sum of two marked polytopes.}\label{fig:summp}
\end{figure}

Marked polytopes appear naturally in many contexts:
\bn
\item Let $N$ be a 3-manifold. A  class $\phi\in H^1(N;\R)$ is called \emph{fibered} if it can be represented by a non-degenerate closed 1-form.
By \cite{Ti70} an integral class $\phi\in H^1(N;\Z)=\hom(\pi_1(N),\Z)$ is fibered if and only if there exists a locally trivial fiber bundle 
$p\co N\to S^1$ such that $p_*=\phi\colon \pi_1(N)\to \pi_1(S^1)=\Z$.
Thurston~\cite{Th86} showed that we can turn $\TT(N)$ into a marked polytope $\MM(N)$ which has the property that $\phi\in H^1(N;\R)=\hom(H_1(N;\R),\R)$ is fibered if and only if it pairs maximally with a marked vertex. This means that there exists a marked vertex $v$ of $\MM(N)$, such that $\phi(v)>\phi(w)$ for any vertex $w\ne v$ in the underlying polytope $\TT(N)$.
\item Let $\Gamma$ be a torsion-free group. We will now see that we can associate a marked polytope to a non-zero element in $\Z[\G]$. (In our subsequent discussion of this assignment we will use the notation introduced in Section~\ref{section:The_polytope_homomorphism}.) We recall that we have an identification $\Z[\Gamma]= \Z[K]*_s H$, where $H=H_1(\Gamma;\Z)/\mbox{torsion}$ and $K=\ker(\Gamma\to H)$.  Given a non-zero element $f=\sum_{h\in H} a_h h$ we consider the marked polytope $\MM(f)$ which is given by the polytope
\[ \PP\Big(\tmsum{h\in H}{} a_h h\Big)\,=\, \mbox{(convex hull of $h$ with $a_h\ne 0$)}\]
where we mark a vertex $h$ of $\PP(f)$ if $a_h=\pm k$ for some $k\in K$. If $\Z[\G]$ is a domain, then  for any non-zero $f,g\in \Z[\Gamma]$ we have $\MM(f\cdot g)= \MM(f)+\MM(g)$.
\item  Let $\pi=\ll x,y|r\rr$ be a nice presentation such that $\Gamma_\pi$ is torsion-free.
We showed in \cite[Proposition~3.5]{FT15} that there exists a unique marked polytope $\MM(\pi)$ with $\MM(\pi)+\MM(y-1)=\MM(\frac{\partial r}{\partial x})$.
\en 
The set of marked polytopes in a vector space forms a monoid under Minkowski sum. As illustrated in Figure~\ref{fig:summp-not-cancel} this monoid does \emph{not} have the cancellation property. This implies that the monoid of marked polytopes does \emph{not} inject into the corresponding Grothendieck group. Therefore we can not associate a meaningful notion of a difference of marked polytopes to an $L^2$-acyclic group $\pi$  of type $F$ at this point of time. 
\begin{figure}[h]
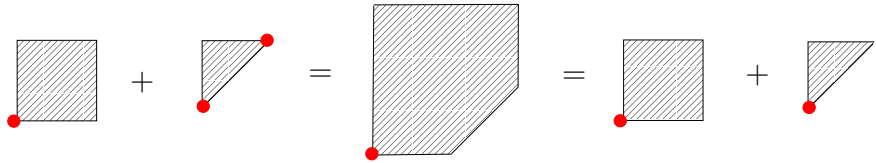
\caption{The Minkowski  sum of marked polytopes does not have the cancellation property.}\label{fig:summp-not-cancel}
\end{figure}

%==================================================================
\section{The polytope invariant and intrinsic properties of the group}
\label{section:polytope-meaning}
In Section~\ref{section:definition-of-p(gamma)} we associated to an $L^2$-acyclic $\whac$-group $\Gamma$  of type $F$ its polytope  invariant $\PP(\Gamma)\in \mfg(\Gamma)$. In this section we want to determine information encoded in this invariant.

%==================================================================
\subsection{The thickness of the polytope invariant and splittings of the group}\label{section:thickness}
Given a polytope $\PP$ in a real vector space $V$ and given $\phi\in  \hom(V,\R)$ we refer to 
\[ \op{th}_\phi(\PP)\,\,=\,\,\max\{ \phi(x)-\phi(y)\,|\, x,y\in \PP\}\,\in \,\R_{\geq 0}\]
as the \emph{thickness of $\PP$ in the $\phi$-direction}.
Since  thickness is translation invariant we can also define $\th_\phi(\PP)\in \R_{\geq 0}$ for any $\PP\in \mfp(V)$ and since thickness is additive under Minkowski sum we can also define $\th_\phi(\PP)\in \R$ for any $\PP\in \mfg(V)$.

Given a polytope $\PP$ in $V$ we refer to
\[ \PP^{\sym}:=\{ \tfrac{1}{2}(p-q)\,|\, p,q\in \PP\}\]
as the \emph{symmetrization of $\PP$}. This definition extends to a symmetrization map on $\mfp(V)$ and $\mfg(V)$. It is clear that the thickness of $\PP$ depends only on the symmetrization of $\PP$. 

The following theorem is now a straightforward consequence of Theorem~\ref{thm:polytopes-agree-for-3-manifold}.

\begin{theorem}\label{thm:thickness-equals-thurston-norm}
For any admissible 3-manifold $N\ne S^1\times D^2$ that is not a closed graph manifold and any $\phi\in H^1(N;\R)$ we have
\[ \th_\phi(\PP(\pi_1(N)))\,\,=\,\, x_N(\phi).\]
\end{theorem}

Ideally one would like to generalize the statement of 
Theorem~\ref{thm:thickness-equals-thurston-norm} to larger classes of groups. The first problem that arises is that there is no satisfactory purely group theoretic definition of the Thurston norm. We refer to \cite{FSW15,FSW16} for several ideas and approaches.

In an attempt to generalize Theorem~\ref{thm:thickness-equals-thurston-norm} we will work with the notion of a splitting of a group. 
Let $\Gamma$ be a finitely presented group and let $\phi\co \Gamma\to \Z$ be an epimorphism. Let $B$ be a finitely generated group. A 
\emph{splitting of $(\Gamma,\phi)$ over  $B$} is an isomorphism
\[ f\co \Gamma\,\,\xrightarrow{\cong} \,\,\ll A,t\,|\, \mu(B)=tBt^{-1}\rr \]
such that the following hold:
\bn
\item $A$ is finitely generated,
\item  $B$ is a subgroup of $A$ and $\mu\co B\to A$ is a monomorphism,
\item  $(\phi \circ f^{-1})(a)=0$ for $a\in A$  and $(\phi\circ f^{-1})(t)=1$.
\en
It is well-known, see e.g.\ \cite{BS78} or \cite[Theorem~B*]{St84}, that any such pair
$(\Gamma,\phi)$ admits a splitting  over a finitely generated group.
We define the \emph{splitting complexity of $(\Gamma,\phi)$} as 
\[ c(\Gamma,\phi)\,\,:=\,\,\min\{ \rank(B)\,|\, (\Gamma,\phi) \mbox{ splits over }B\},\]
where $\rank(B)$ is defined as the minimal number of generators of $B$.

\begin{theorem}\label{thm:splittings}
Let  $\Gamma$ be a torsion-free group  that admits a nice presentation $\ll x,y|r\rr$. Furthermore suppose that $\Gamma$ satisfies the Atiyah Conjecture. Then for any epimorphism $\phi\co \Gamma\to \Z$ we have
\[ c(\Gamma,\phi)-1\,\,=\,\,\th_\phi(\PP(\Gamma)).\]
\end{theorem}

\begin{remark}
\bn
\item
In \cite{FSW15} it was shown that the statement of the theorem also holds if $\pi$ is the fundamental group of a knot complement. But the equality of Theorem~\ref{thm:splittings} does not hold for the fundamental groups of all 3-manifolds. For example, if $N$ is the 3-torus, then $\pi_1(N)=\Z^3$ and it follows from
Theorem~\ref{thm:properties-of-p} and the discussion in Section~\ref{section:elementary-examples} 
 that $\PP(\pi_1(N))$ is  a point. In particular the thickness for each $\phi$ is zero. On the other hand it is straightforward to see that for any epimorphism $\phi\colon \pi_1(N)\to \Z$ we have $c(\pi_1(N),\phi)=2$.
  \item The inequality $c(\Gamma,\phi)-1\geq \th_\phi(\PP(\Gamma))$ was shown in \cite[Proposition~7.6]{FT15} if $\Gamma$ is residually a group that is elementary-amenable and torsion-free (without the assumption that $\Gamma$ satisfies the Atiyah Conjecture). As we remarked after Theorem~\ref{thm:polytopes-for-2-1-presentations}, potentially most torsion-free groups with a nice presentation satisfy this condition.
 \en
\end{remark}

\begin{proof}
The inequality $c(\Gamma,\phi)-1\leq\th_\phi(\PP(\Gamma))$ was shown in \cite[Proposition~7.3]{FT15}. As was suggested to us by Nathan Dunfield, the proof of this inequality follows from a careful reading of the first step in the proof of the Freiheitssatz.

The proof of the reverse inequality  follows along the same lines as the proof \cite[Proposition~7.6]{FT15} but one needs to replace the Ore localizations of the elementary-amenable and torsion-free groups by  $\DD(\Gamma)$.
More precisely, it follows easily from the combination of Theorem~3.6(4) and Lemmas~4.3, 6.12 and 6.13 of \cite{FL16a}  together with a slight variation on Theorem~4.1 of \cite{FL16a}. We leave the details to the reader.
\end{proof}

%==================================================================
\subsection{The polytope invariant and its relation to the BNS-invariant}

Let $\Gamma$ be a finitely generated group. The Bieri--Neumann--Strebel \cite{BNS87} invariant $\S(\Gamma)$ of $\Gamma$ is by definition a subset of $S(\Gamma):=(\hom(\Gamma,\R)\sm \{0\})/\R_{>0}$. We refer to \cite{BNS87} for the precise definition,
but in order to give a flavour of the invariant we recall three properties:
\bn
\item An epimorphism $\phi\in \hom(\Gamma,\Z)$ represents 
an element in $\S(\Gamma)$ if and only if it corresponds to an \emph{ascending HNN-extension}. To be precise: if and only if there exists an isomorphism
\[ f\co \Gamma\,\,\to\,\, \ll A,t\,|\, A=t^{-1}\varphi(A)t\rr\]
where $A$ is a finitely generated group and $\varphi\co A\to A$ is a monomorphism, such that $\phi$ corresponds under $f$ to the epimorphism given by $t\mapsto 1$ and $a\mapsto 0$ for all $a\in A$.\footnote{Some authors, e.g.\ \cite{FK16}, call this a \emph{descending} HNN-extension. We follow the convention used in \cite{BS78,BNS87}.}
\item A non-trivial homomorphism $\phi\in \hom(\Gamma,\Z)$ has the property that $\phi$ and $-\phi$ represent elements in $\S(\Gamma)$
if and only if $\ker(\phi)$ is finitely generated.
\item $\Sigma(\Gamma)$ is an open subset of $S(\Gamma)$.
\en
The first two properties follow from \cite[Proposition~4.3]{BNS87} (see also \cite[Corollary~3.2]{Br87}) and the third one is \cite[Theorem~A]{BNS87}.

It is shown in \cite[Theorem~E]{BNS87} that given a 3-manifold $N$ the fibered classes of $H^1(N;\R)=\hom(\pi_1(N),\R)$ correspond precisely to the classes that lie in $\Sigma(\pi_1(N))$. The following is now a reformulation of the statement of the first example in Section~\ref{section:marked-polytopes}.

\begin{theorem}
Let $N\ne S^1\times D^2$ be an admissible 3-manifold that is not a closed graph manifold. Then the polytope $\PP(N)$ admits a marking  with the property that for any non-trivial $\phi\in \hom(\pi_1(N),\R)$ we have
\[ [\phi]\in \Sigma(\pi_1(N))\quad \Longleftrightarrow \quad \phi\mbox{ pairs maximally with a marked vertex of $\PP(N)$.}\]
\end{theorem}

For groups with a nice presentation $\ll x,y|r\rr$ we obtain a similar theorem.

\begin{theorem}\label{mainthm}
Let $\pi=\ll x,y\,|\, r\rr$ be a nice presentation.
A non-trivial class $\phi\in H^1(\G_\pi;\R)$  represents an element in $\Sigma(\G_\pi)$ if and only if $\phi$ pairs maximally with a marked vertex of $\MM(\pi)$. 
\end{theorem} 

The theorem is stated as  \cite[Theorem~1.1]{FT15}. In that paper the theorem is proved using  the generalised Novikov rings of Sikorav~\cite{Sik87}. It can also be viewed  as a reformulation of Brown's algorithm \cite[Theorem~4.3]{Br87}.

Finally a relationship between the polytope invariant and the BNS-invariant was also found for a different class of groups by Funke--Kielak~\cite{FK16}.
In order to state their theorem we need to introduce a few more definitions.
\bn
\item Let $\varphi\colon \pi\to \pi$ be a monomorphism of a group $\pi$. We denote by
\[ \pi*_\varphi\,\,:=\,\, \ll \pi,t \,|\, t^{-1}gt=\varphi(g), g\in \pi\rr\]
the corresponding descending HNN-extension.\footnote{In \cite{FK16}  we authors refer to such an HNN-extension as an ``ascending HNN-extension'', whereas here we stick with the convention eastablished above and that follows \cite{BS78,BNS87}.} We refer to the epimorphism $\pi*_\varphi\to \Z$ given by $g\mapsto 0$ for all $g\in \pi$ and $t\to 1$ as the \emph{canonical epimorphism}.
\item Let $\PP$ be a polytope in a vector space $V$ and let $\phi\in \hom(V,\R)$. Following~\cite{FK16} we refer to 
\[ F_\phi(\PP)\,\,=\,\,\{v\in \PP\,|\, \phi(v)\leq \phi(w)\mbox{ for all }w\in \PP\}\]
as the \emph{$\phi$-minimal face of $\PP$}. 
\item Let $V$ be a vector space and let $\SS=[\PP]-[\QQ]\in \mfg(V)$. Following~\cite{FK16} we  say that $\phi,\psi\in \hom(V,\R)$ are \emph{$\SS$-equivalent} if 
$F_\phi(\PP)=F_\psi(\PP)$ and $F_\phi(\QQ)=F_\psi(\QQ)$. Note that this definition is independent of the choice of $\PP$ and $\QQ$.

\en

The following is the main result of \cite{FK16}.

\begin{theorem}\label{thm:funke-kielak}
Let $\varphi\colon F_2\to F_2$ be a monomorphism of the free group on two generators. Let $\phi\colon \Gamma:=F_2*_\varphi\to \R$ be a homomorphism such that the class $[\phi]\in S(\Gamma)$ is not represented by the canonical epimorphism. Then there exists an open neighborhood $U$ of $[\phi]$ in $S(\Gamma)$ such that for every non-trivial $\psi\colon \Gamma\to \R$ with
$[\psi]\in U$ that is $\PP(\Gamma)$-equivalent to $\phi$ we have
\[ [-\phi]\in \Sigma(\Gamma)\quad \Longleftrightarrow \quad [-\psi]\in \Sigma(\Gamma).\]
\end{theorem}

%==================================================================
\section{Questions}\label{section:questions}
We conclude this survey with a long list of questions and conjectures.

\begin{conjecture}\label{conj:free-by-cyclic}
If $\Gamma$ is a free-by-cyclic group, then $\PP(\Gamma)\in \mfg(\Gamma)$ can be represented by a polytope.
\end{conjecture}

\begin{conjecture}\label{conj:2-dimensional-k(pi,1)}
Let $\Gamma\ne \Z$ be an $L^2$-acyclic $\whac$-group  of type $F$ that  admits a 2-dimensional $K(\Gamma,1)$.  Can $\PP(\Gamma)\in \mfg(\Gamma)$ be represented by a polytope? 
\end{conjecture}

A proof of Conjecture~\ref{conj:2-dimensional-k(pi,1)} also proves Conjecture~\ref{conj:free-by-cyclic}.

Before we state our next question we recall that  the Thurston norm of an admissible hyperbolic 3-manifold $N$ is a norm. (This  is a direct consequence of the fact that hyperbolic admissible 3-manifolds are atoroidal, i.e.\ any embedded torus is boundary parallel, see e.g.\ \cite[Proposition~D.3.2.8]{BP92}.)
This implies that if $N$ is an admissible hyperbolic 3-manifold with $b_1(N)\geq 1$, then $\PP(N)=\PP(\pi_1(N))$ is not a point.

\begin{question}\label{qu:2-dimensional-k(pi,1)-hyperbolic}
Let $\Gamma\ne \Z$ be an $L^2$-acyclic $\whac$-group  of type $F$ that  admits a 2-dimensional $K(\Gamma,1)$. Furthermore suppose that $\Gamma$ is a non-elementary hyperbolic group and that $b_1(\Gamma)\geq 1$. Does it follow that $\PP(\Gamma)\ne 0\in \mfg(\Gamma)$?
\end{question}

\begin{conjecture}
Let $\Gamma$ be an $L^2$-acyclic $\whac$-group  of type $F$. If $\Gamma$ is  amenable and if $\Gamma$ is not virtually $\Z$, then $\PP(\Gamma)=0\in \mfg(\Gamma)$.
\end{conjecture}

\begin{question}\label{question:polytope-bns}
Let $\Gamma$ be an $L^2$-acyclic $\whac$-group  of type $F$. Is the BNS-invariant related to $\PP(\Gamma)$?  
Does an analogue of Theorem~\ref{thm:funke-kielak} hold?
\end{question}

In general it might be too optimistic to expect a positive answer. It seems more likely that the question can be answered in the affirmative if $\Gamma$ has a 2-dimensional $K(\Gamma,1)$.

It is known that the BNS invariants of metabelian and 3--manifold groups are polyhedral. An affirmative (even partial) answer to Question~\ref{question:polytope-bns} may establish polyhedrality of BNS invariants for new families of groups, as well as whether the polyhedra are rational or not.
We can therefore also ask the following question.

\begin{conjecture}
Let $\Gamma$ be a hyperbolic $L^2$-acyclic group  of type $F$ with a 2-dimensional $K(\Gamma,1)$. Then there exists an integral marked polytope $\MM\subset H_1(\Gamma;\R)$ such that a non-trivial $\phi\in H^1(\Gamma;\R)$ represents an element in $\Sigma(\Gamma)$ if and only if $\phi$ pairs maximally with a marked vertex of $\MM$.
\end{conjecture}

The combination of Theorem~3.6(4) and Lemmas~4.3, 6.12 and 6.13 of \cite{FL16a}  shows that if $\Gamma$ is an $L^2$-acyclic $\whac$-group  of type $F$, then for any epimorphism $\phi\colon \Gamma\to \Z$ we have
\[ \smsum{n\in \N_0}{}(-1)^{n+1}\cdot  b_n^{(2)}(\ker(\phi\colon \Gamma\to \Z))\,\,=\,\,\th_\phi(\PP(\Gamma)).\]

\begin{conjecture}
Let $\Gamma$ be  an $L^2$-acyclic $\whac$-group  of type $F$ and let $n\in \N_0$. Then there exists a polytope $\PP$ such that 
\[  b_n^{(2)}(\ker(\phi\colon \Gamma\to \Z))\,\,=\,\,\th_\phi(\PP)\]
for any epimorphism $\phi\colon \Gamma\to \Z$.
\end{conjecture}

In Section~\ref{section:thickness} we saw that for 3-manifold groups and many two-generator one-relator  groups the thickness of the polytope  invariant is related to the complexity of splittings of the underlying group.  

\begin{question}
Let $\Gamma$ be an $L^2$-acyclic $\whac$-group  of type $F$. What information does the thickness of the polytope invariant contain?
\end{question}

As we mentioned before, the thickness of a polytope is an invariant of the symmetrized polytope. 

\begin{question}
What information does the polytope invariant contain, that cannot be obtained from the symmetrized polytope invariant?
\end{question}

One partial answer is given by the discussion of the BNS-invariant. This invariant is in general not symmetric and if the polytope invariant is related to the BNS-invariant it cannot be symmetric in general. Nonetheless, there are many groups with empty BNS-invariant and non-symmetric polytope invariant.

Finally we want to discuss which elements of  $\mfg(\Z^n)$ can be realized as the polytope invariant of a manifold. Here we say that $\PP\in \mfg(\Z^n)$ \emph{can be realized by a $d$-dimensional manifold} if there  exists a pair $(N,\varphi)$ where $N$ is a closed $L^2$-acyclic $d$-dimensional manifold $N$ such that $\pi_1(N)$ is a $\whac$-group and where $\varphi\colon \Z^n\to H_1(N;\Z)/\mbox{torsion}$ is an isomorphism, such that 
$\varphi_*(\PP)=\PP(N)\in \mfg(\pi_1(N))$. 

From our above results we know that not all $\PP\in \mfg(\Z^n)$ can be realized by manifolds. More precisely, in Proposition~\ref{prop:properties-of-p-2} we showed that polytopes realized by closed orientable manifolds have a symmetry.
Furthermore in Theorem~\ref{thm:polytopes-agree-for-3-manifold} we showed that the polytope invariant of an admissible 3-manifold $N\ne S^1\times D^2$ that is not a closed graph manifold can  be represented by a polytope. 
So we can only hope to realize elements of $\mfp(\Z^n)\subset \mfg(\Z^n)$ by 3-dimensional manifolds.

We have the following realization result.

\begin{lemma}\label{lem:realizibility}
Let $n\in \N$  and let $\PP \in \mfg(\Z^n)$.
For any $d\geq 7$ we can realize $\PP+(-1)^{d+1}\ol{\PP}$ by a $d$-dimensional manifold.
\end{lemma}

\begin{proof}
Let $F$ be a free group on  $n-1$ generators. We set $\Gamma=F\times \Z$.
Then $K(\Gamma,1)$ is 2-dimensional and it is $L^2$-acyclic by \cite[Theorem~1.35~(4)]{Lu02}. Furthermore $\Gamma$ is a $\whac$-group by \cite[Theorem~2]{KLR16} and  \cite[Theorem~2.3]{LiL16}. 
We pick an isomorphism $\varphi\colon \Z^n\to H_1(\Gamma;\Z)$
and we pick a set-theoretic section $s\colon H_1(\Gamma;\Z)\to \Gamma$ of the projection map $\Gamma\to H_1(\Gamma;\Z)$.

Now let $\PP\in \mfg(\Z^n)$. We can find $a,b\in \Z[\Z^n]$ with $\PP=\PP(a)-\PP(b)$. 
%We set $\QQ=\PP(s(\varphi(a)))-\PP(s(\varphi(b)))\in \mfg(\Gamma)$. Evidently we have $\varphi(\PP)=\QQ$. 
We write 
\[ \omega\,:=\,s(\varphi(a))\cdot s(\varphi(b))^{-1}\,\in\, K_1(\DD(\Gamma))=\DD(\Gamma)^\times_{ab}.\]
It follows immediately from the definitions that  $\PP(\omega)=\varphi(\PP)\in \mfg(\Gamma)$. 
Now let $d\geq 7$. It follows easily from a slight generalization\ of  \cite[Lemma~2.9]{FL16b} that there exists a closed orientable $d$-dimensional manifold $N$ with $\pi_1(N)=\Gamma$ such that 
\[\tau(N,\id_{\pi_1{N}})\,=\,\,\omega\cdot \ol{\omega}^{(-1)^{d+1}}\,\in\, K_1(\DD(\Gamma))=\DD(\Gamma)^\times_{ab}/\pm \Gamma. \]
It follows immediately that $\PP(N)=\varphi(\PP)+(-1)^{d+1}\ol{\varphi(\PP)}\in\mfg(\Gamma)$. 
\end{proof}

\begin{question}
\bn[font=\normalfont]
\item Let $\PP\in \mfp(\Z^n)$. Does there exist a closed orientable admissible 3-manifold  that realizes $\PP+\ol{\PP}$? 
\item 
Let $\PP\in \mfg(\Z^n)$ and let $d\in \{4,5,6\}$. Does there exist an $L^2$-acyclic $d$-dimensional 
closed orientable manifold that realizes $\PP+(-1)^{d+1}\ol{\PP}$?
\en
\end{question}

The first question seems to be very hard since by 
Theorem~\ref{thm:polytopes-agree-for-3-manifold}  it is a reformulation of the question of which Thurston norm balls are realized by aspherical 3-manifolds. This question has been open since the 1970's. 
The second question might be more accessible, especially in dimensions 5 and 6. It is conceivable that in dimension 4 the answer depends on whether one studies topological or smooth manifolds.

\end{document}